\documentclass{amsart}
\usepackage{latexsym,amssymb,amsmath,amsthm,latexsym}
\usepackage{amsfonts}
\theoremstyle{definition}
\newtheorem{theorem}{Theorem}

\newtheorem{corollary}{Corollary}

\newtheorem{proposition}{Proposition}

\newtheorem{remark}{Remark}

\newtheorem{identity}{Identity}
\title{On Enumeration of Dyck Paths with Colored Hills}
\author{Milan Janji\'c}
\date{\today}
\begin{document}
\maketitle
\begin{center}Department for Mathematics and Informatics, University of Banja
Luka,\end{center}
 \begin{center}Republic of Srpska, BA\end{center}
\begin{abstract} We continue to investigate the properties of the earlier defined functions
 $f_m$ and $g_m$, which depend on an initial arithmetic function $f_0$.
 In this papers values of $f_0$ are the Fine numbers.  We investigate functions $f_i,g_i,(i=1,2,3,4)$. For each function, we derive an explicit formula and give a combinatorial interpretation.
 It appears that  $g_2$ and $g_3$ are well-known combinatorial object called the Catalan triangles.

 We finish with an identity consisting of ten items.
\end{abstract}

\section{Introduction} This paper is a  continuation of the investigations of restricted
words from the author's  previous papers, where two quantities $f_m(n)$ and $g_m(n,k)$ are
defined as follows. For an initial arithmetic function $f_0$, $f_m,(m>1)$ is the $m$th
invert transform of $f_0$. The function $g_m(n,k)$ is defined  in the following
way:
\begin{equation}\label{cmnk}g_m(n,k)=\sum_{i_1+i_2+\cdots+i_k=n}f_{m-1}(i_1)\cdot
f_{m-1}(i_2)\cdots f_{m-1}(i_k).\end{equation}
Also, the following equation holds:
\begin{equation}\label{ffm}
f_m(n)=\sum_{k=1}^ng_m(n,k).
\end{equation}
We restate~\cite [Propositions 10]{ja3} which will be used throughout the paper.
\begin{proposition}\label{prr}
  Let $f_0$  the arithmetic function which values are nonnegative integers, and $f_{0}(1)=1$. Assume next  that, for $n\geq 1$, we have $f_{m-1}(n)$
words of length $n-1$ over a finite alphabet $\alpha$. Let  $x$ be a letter which is not
in $\alpha$. Then, the value of $g_m(n,k)$ is the number of words of length $n-1$ over
the alphabet $\alpha\cup\{x\}$ in which $x$ appears exactly $k-1$ times.
\end{proposition}
We denote by $G_m(n)$ the array $g_m(n,k)$ viewed  as a lower triangular matrix of order $n$.
It is proved in~\cite[Proposition 6]{ja3} that
\begin{equation}\label{CM}G_m(n)=G_1(n)\cdot L_n^{m-1}.\end{equation}
In particular, we have
\begin{equation}\label{gsum}g_m(n,k)=\sum_{i=k}^{n}{i-1\choose k-1}c_{m-1}(n,i),
\end{equation}
and
\begin{equation}\label{gf}\sum_{n=k}^\infty g_m(n,k)x^n=\left(\sum_{i=1}^{\infty}f_{m-1}(i)x^i\right)^k.
\end{equation}
\section{ A combinatorial result}
We start with a combinatorial proof that the sequence $C_0,C_1,\ldots$ of the Catalan numbers  is the invert transform of the sequence $\mathbb F_1,\mathbb F_2, \ldots$ of the Fine numbers with $\mathbb F_1=1$.
We define $f_0(n)=\mathbb F_{n},(n\geq 1)$.
 Thus, $f_0(1)=1,f_0(2)=0,f_0(3)=1,$ and so on.  It is a well-known fact that  $\mathbb
 F_{n}$ is the number of Dyck paths  of semi-length $n-1$ with no hills.

All investigation in the paper are based on the following result.
\begin{theorem}\label{te} For $m\geq 1$, we have
\begin{enumerate}
\item The value of $g_m(n,k)$ is the number of Dyck paths of semilength $n-1$ having hills in $m$ colors, of which  $k-1$ are in color $m$.
\item The value of $f_m(n)$ is the number of Dyck paths of semilength $n-1$ having hills in $m$ colors.
\end{enumerate}
\end{theorem}
\begin{proof} We use induction on $m$. We have $f_0(n)=\mathbb F_n$. It is well-known that
$f_0(n)$ is the number of Dyck paths of semilength $n-1$ with no hills. If we consider the symbol $x$ in Proposition \ref{prr} as a hill (of color $1$), then the first  assertion holds for $m=1$.
The second assertion holds by (\ref{ffm}).

Assume that the assertion is true for $m>1$, that is,
Assume that $f_{m-1}(n)$ equals the number  of Dyck paths of semilength $n-1$ having hills in $m-1$ colors. Since $f_{m-1}(1)=1$ and since the empty Dyck path has no hills, we may apply (\ref{ffm}) to obtain the assertion.
\end{proof}

We state two particular cases. Firstly, for $m=1$, the value of $f_1(n)$ is the number of Dyck paths of semilength $n-1$, which equals the Catalan number $C_{n-1}$.
Hence,
\begin{equation*}f_1(n)=C_{n-1}.\end{equation*}

Since  $f_1(1),f_1(2),\ldots$ is the invert transform of $f_0(1),f_0(2),\ldots$, we obtain the following relation between Fine and Catalan numbers.
\begin{corollary} The sequence $C_0,C_1,\ldots$ of the Catalan numbers is the invert transform of the sequence $\mathbb F_1,\mathbb F_2,\ldots$ of the Fine numbers.
\end{corollary}
From~\cite[Identity 12]{ja3}, by the use of the identity  $i\cdot{i-1\choose k-1}=k{i\choose k}$, we obtain the following identity relating the Fine and the Catalan numbers via the partial Bell polynomials.
\begin{identity}
\begin{equation*}(k-1)!B_{n,k}(C_0,2!\cdot C_1,3!\cdot C_2,\ldots)=\sum_{i=k}^n{i\choose
k}(i-1)!B_{n,k}({\mathbb F}_1,2!\cdot {\mathbb F}_2,3!\cdot {\mathbb F}_3,\ldots).
\end{equation*}
\end{identity}
\section{Catalan triangle $g_2(n,k)$}
According to Theorem \ref{te}, we have
\begin{corollary}
\begin{enumerate}
\item
The value of $g_1(n,k)$ is the number of Dyck paths of semilength $n-1$ having $k-1$ hills.
\item The number of Dyck paths of semilngth $n-1$ equals $f_1(n)$.
\end{enumerate}
\end{corollary}
We  conclude that
\begin{equation*}f_1(n)=C_{n-1},\end{equation*} where $C_{n-1}$ is the Catalan number.

 An explicit formula for $g_1(n,k)$ will be derived later.

The Segner's formula means that the sequence $C_1,C_2,\ldots$ of the Catalan numbers is the invert transform of the sequence $C_0,C_1,\ldots$. This yields that $f_2(n)=C_n$, for all $n$.
We thus obtain the following combinatorial interpretation of the Catalan numbers.
\begin{corollary} The Catalan number $C_n$ is the number of Dyck paths of semilength $n-1$ having hills in two colors.
\end{corollary}
\begin{remark} Note that this property of Catalan number is equivalent to Stanley~\cite[BE.52]{rst}.
\end{remark}
We also have  the following identity relating the Catalan numbers and the partial Bell polynomials.
\begin{identity}
\begin{equation*}(k-1)!B_{n,k}(C_1,2!\cdot C_1,3!\cdot C_3,\ldots)=\sum_{i=k}^n{i\choose
k}(i-1)!B_{n,k}(C_0,2!\cdot C_1,3!\cdot C_2,\ldots).
\end{equation*}
\end{identity}

We firstly derive an explicit formula for $g_2(n,k)$.
It is known that \begin{equation*}g(x)=\frac{1-\sqrt{1-4x}}{2x}\end{equation*} is the ordinary generating function for the sequence $C_0,C_1,\ldots$. It follows from (\ref{gf}) that
$\sum_{n=k}^\infty g_2(n,k)x^n$ is the expansion of $[xg(x)]^k$ into powers of $x$. Using the binomial theorem and the expansion of a binomial series, we obtain
\begin{equation*}\sum_{n=k}^\infty g_2(n,k)x^n=
\sum_{j=0}^\infty\left[\sum_{i=0}^k(-1)^{i+j}{k\choose i}\frac{\prod_{t=0}^{j-1}(i-2t)}{2^{k-j}\cdot j!}\right]x^j.
\end{equation*}
Comparing coefficients of the same powers of $x$, we firstly obtain
\begin{identity} If $n<k$, then
\begin{equation*}\sum_{i=0}^k(-1)^{i+n}{k\choose i}\prod_{t=0}^{n-1}(i-2t)=0.\end{equation*}
\end{identity}

The case $n\geq k$ yields

\begin{equation*}g_2(n,k)=\frac{2^{n-k}}{n!}\sum_{i=1}^k(-1)^{i+n}{k\choose i}\prod_{t=0}^{n-1}(i-2t).
\end{equation*}

It is clear that $\prod_{t=0}^{n-1}(i-2t)=0$, if $i$ is even. If $i$ is odd, we denote $i=2j-1,(1\leq j\leq\lfloor\frac{k+1}{2}\rfloor$. Hence
\begin{proposition} The following formula holds:
\begin{equation}\label{pr1}
g_2(n,k)=\frac{2^{n-k}}{n!}\sum_{j=1}^{\lfloor\frac{k+1}{2}\rfloor}
(-1)^{j-1}{k\choose 2j-1}\cdot (2j-1)!!\cdot(2n-2j-1)!!.
\end{equation}
\end{proposition}
In particular, since $g_2(n,1)=f_1(n)=C_{n-1}$, we obtain the following result.
\begin{corollary}For $n>1$, we have
\begin{equation*}C_{n-1}=\frac{2^{n-1}}{n!}\cdot(2n-3)!!.
\end{equation*}
\end{corollary}
\begin{remark} The preceding is the famous Euler's formula for the Catalan numbers.
\end{remark}

We now prove that $g_2(n,k)$ satisfies a simple recurrence relation.
\begin{proposition}\label{rf} For $1\leq k<n$, the following recurrence holds:
\begin{equation}\label{cik}g_2(n+1,k+1)=g_2(n+1,k+2)+g_2(n,k).\end{equation}
\end{proposition}
\begin{proof} According to (\ref{cmnk}), we have
\begin{equation}\label{drc}g_2(n+1,k+1)=\sum_{i_1+i_2+\cdots+i_{k+1}=n+1}C_{i_1-1}\cdot C_{i_2-1}\cdots C_{i_{k+1}-1},\end{equation}
where the sum is taken over positive $i_t$.

Firstly, we extract the term obtained for $i_{k+1}=1$.
Since  $C_{i_{k+1}-1}=C_0=1$, we obtain
\[\sum_{i_1+i_2+\cdots i_k=n}C_{i_1-1}\cdot C_{i_2-1}\cdots C_{i_{k}-1}=C_2(n,k),\]
which is the second term on the right-hand side in formula (\ref{rf}).
It remains to calculate the sum on the right-hand side of Equation (\ref{drc}), when
$i_{k+1}>1$.
We consider the equation
\begin{equation*}g_2(n+1,k+2)=\sum_{j_1+j_2+\cdots+j_{k+1}+j_{k+2}=n+1}C_{j_1-1}\cdot C_{j_2-1}\cdots C_{j_{k+1}-1}\cdot C_{j_{k+2}-1}.\end{equation*}  Denote $j_{k+1}+j_{k+2}=i_{k+1}>1$.
This equation is fulfilled for the following pairs of $(j_{k+1},j_{k+2})$:
\begin{equation*}\{(1,i_{k+1}-1),(2,i_{k+1}-2),\ldots,(i_{k+1}-1,1)\}.\end{equation*}.
 We rearrange terms in
the sum as follows:
\[g_2(n+1,k+2)=\sum_{j_1+\cdots+j_k+i_{k+1}=n+1}C_{j_1-1}C_{j_2-1}\cdots
C_{j_{k}-1}\cdot\sum_{i=1}^{i_{k+1}-1}C_{i-1}C_{i_{k+1}-1-i}.\]
Segner's formula implies  $\sum_{i=1}^{i_{k+1}-1}C_{i-1}C_{i_{k+2}-1-i}=C_{i_{k+1}-1}$. We
thus obtain
\[g_2(n+1,k+2)=\sum_{j_1+\cdots+j_k+i_{k+1}=n+1}C_{j_1-1}C_{j_2-1}\cdots
C_{j_{k}-1}\cdot C_{i_{k+1}-1},\] for $i_{k+1}>1$, which is the first term in Equation (\ref{cik}).
\end{proof}
We now prove that the following vertical recurrence holds:
\begin{proposition} For $n,k>1$, we have
\begin{equation}\label{rr1}
g_2(n,k)=\sum_{i=k-1}^{n-2}g_2(n-1,i)+1.
\end{equation}
\end{proposition}
\begin{proof}
From Proposition \ref{rf}, we obtain the following sequence of equations.
\begin{align*}
g_2(n+1,3)&=g_2(n+1,2)-g_2(n,1),\\
g_2(n+1,4)&=g_2(n+1,3)-g_2(n,2),\\&\vdots\\
g_2(n+1,k+2)&=g_2(n+1,k+1)-g_2(n,k).
\end{align*}
Adding terms on the left-hand sides of these equations, and those on the right-hand sides, we obtain
\begin{equation*}
\sum_{i=1}^kg_2(n,i)=g_2(n+1,2)-g_2(n+1,k+2).
\end{equation*}
Replacing $n$ by $n-1$, and $k$ by $k-2,(k>2)$, we obtain
\begin{equation*}g_2(n,2)=g_2(n,k)+\sum_{i=1}^{k-2}g_2(n-1,i).\end{equation*}
In particular, for $k=n$, this equation becomes
\begin{equation*}g_2(n,2)=\sum_{i=1}^{n-2}g_2(n-1,i)+1,
\end{equation*}
and the formula follows.
\end{proof}
We now derive a simpler explicit formula for $g_2(n,k)$.
\begin{proposition} The following formula holds:
$g_2(n,n)=1$, and
\begin{equation}\label{g2nk}g_2(n,k)=
\frac{k\prod_{i=1}^{n-k-1}(n+i)}{(n-k)!},
\end{equation}
otherwise.
Equivalently,
\begin{equation}\label{g2}g_2(n,k)=\frac{k}{n-k}{2n-k-1\choose n},(n>k).
\end{equation}
\end{proposition}
\begin{proof} We use the recurrence (\ref{cik}). We have
\begin{align*}
g_2(n+1,k+1)-g_2(n,k)&=\frac{(n+2)\cdots(2n-k-1)\cdot[(k+1)(2n-k)-k(n+1)]}{(n-k)!}\\
&=\frac{(n+2)\cdots(2n-k-1)\cdot(k+2)(n-k)}{(n-k)!} \\
&=\frac{(k+2)(n+2)\cdots(2n-k-1)}{(n-k-1)!}=g_2(n+1,k+2).
\end{align*}
and the assertion follows by induction.
\end{proof}
From (\ref{rr1}), we obtain the following  identity:
\begin{identity} For $n>k$, we have
\begin{equation*}\frac{k}{n-k}\cdot{2n-k-1\choose n}=\sum_{i=k-1}^{n-2}\frac{i}{n-i-1}{2n-3-i\choose n-1}+1.\end{equation*}
\end{identity}

We denote by $A(n,k)$ the mirror triangle of $g_2(n,k)$. Hence, $A(n,k)=g_2(n,n-k+1)$.
\begin{proposition}
The triangle $A(n,k)$ satisfies the following conditions:
\begin{enumerate}
\item $A(n,1)=1, A(n,n)=C_{n-1}$.
\item $A(n+1,k+1)=A(n+1,k)+A(n,k+1)$,
    \item $A(n,n-1)=C_{n-1}$.
\end{enumerate}
\end{proposition}
\begin{proof}
\begin{enumerate}
\item  We have $A(n+1,1)=g_2(n+1,n+1)=f_0(1)^{n+1}=1$.
Also, $A(n,n)=g_2(n,1)=C_{n-1}$.
\item We have $A(n+1,k+1)=g_2(n+1,n-k+1)$. Using Proposition \ref{rf} yields
\begin{equation*}A(n+1,k+1)=g_2(n+1,n-k+2)+g_2(n,n-k)=A(n+1,k)+A(n,k+1).\end{equation*}
\item
We have $A(n,n-1)=g_2(n,2)$.  According to (\ref{cmnk}), we have
$g_2(n,2)=\sum_{i=1}^{n-2}C_{i-1}C_{n-i-2}$.
Applying the Segner's formula yields $A(n,n-1)=C_{n-1}$.
\end{enumerate}
\end{proof}
\begin{remark}
We note that the triangle $A(n,k)$ is the Catalan triangle, considered in~Koshy\cite[Chapter 15]{tk}. The chapter is devoted to a family of binary words.
\end{remark}
Comparing result which is obtained in this Chapter, and our result, we obtain the following   result:
\begin{identity}The following sets has the same number of elements:
\begin{enumerate}
\item The number of Dyck paths of semilength $n-1$ having hills in two  colors, of which $n-k$ hills in color $2$.
\item The number of binary words of length $n+k-2$ having $n-1$ ones and $k-1$ zeros and no initial segment has more zeros than ones.
\end{enumerate}
\end{identity}
We also  give a bijective proof.
\begin{proof} In a Dyck word of semilength $n-1$ with $n-k$ hills of color $2$, we replace each hill of color $2$ by $1$. Between two hills of color $2$ are the standard Dyck paths, which we interpret as binary words having the same number of zeros and ones, and no initial segment having more zeros that ones. In this way we obtain a binary words having $n-1$ ones and $k-1$ zeros, and no initial segment has more zeros then ones. It is clear that this correspondence is injective.

Conversely, if $w$ is a binary word satisfying 2.
Scanning from left to right, starting from the last zero, we find interval consisting of the same numbers of ones and zeros. This interval produces a standard Dyck path. If there ones behind the last zero, we replace each of them by the hill of color $2$.

Continuing the same procedure, we obtain  Dyck path of semilength $n-1$ having $n-k$ hills of color $2$. This correspondence is also bijective.

\end{proof}
\section{Catalan triangle $g_3(n,k)$}
From Theorem \ref{te}, we obtain
\begin{corollary}
\begin{enumerate}
\item
The value of $g_3(n,k)$ is the number of Dyck paths of semilength $n-1$ hills i three colors, of which $k-1$ hills in color $3$.
\item The number of Dyck paths of semilngth $n-1$ having hills in three colors equals $f_3(n)$.
\end{enumerate}
\end{corollary}
 We derive an explicit formula for $g_3(n,k)$.
\begin{proposition}We have
\begin{equation}\label{bnk}
g_3(n,k)=\frac{k}{n}{2n\choose n-k}.
\end{equation}
\end{proposition}
\begin{proof} According to (\ref{ffm}), we have
\begin{equation*} g_3(n,k)=\sum_{i=k}^n{i-1\choose k-1}g_2(n,i).
\end{equation*}
Using (\ref{g2nk}) yields
\begin{equation*}
g_3(n,k)=\sum_{i=k}^n{i-1\choose k-1}\frac{i\cdot(n+1)\cdot(n+2)\cdots (2n-i-1)}{(n-i)!}
\end{equation*}
Using the identity $i\cdot{i-1\choose k-1}=k{i\choose k}$ implies
\begin{equation*}
g_3(n,k)=\frac kn\cdot\sum_{i=k}^n{i\choose k}\frac{n(n+1)\cdot(n+2)\cdots (2n-i-1)}{(n-i)!},
\end{equation*}
that is
\begin{equation*}
g_3(n,k)=\frac kn\cdot\sum_{i=k}^n{i\choose k}{2n-i-1\choose n-1}.
\end{equation*}
Hence, our statement is equivalent to the following binomial identity.
\end{proof}
\begin{identity} We have
\begin{equation}\label{bi}{2n\choose n+k}=\sum_{i=k}^n{i\choose k}{2n-i-1\choose n-1}.\end{equation}
\end{identity}
\begin{proof}
We prove the identity combinatorially. We count subsets of $n+k$ elements of the set width $2n$ element after the position of the $(k+1)$th element. If $i+1$ is the position of the $k+1$th element of the set then we have ${i\choose k}$ elements before and ${2n-i-1\choose n-1}$ after this element. Since $k\leq i\leq n$ the identity is true.
\end{proof}
\begin{remark}
The Catalan triangle $g_3(n,k)$ is defined by Shapiro~\cite{wsh}.
\end{remark}
Taking into account his original combinatorial interpretation, we obtain the following:
\begin{identity} The following sets has the same number of elements:
\begin{enumerate}
\item The number of nonintersecting lattice paths
in the first quadrant at the distance $k$.
\item The number of Dyck paths of semilength $n-1$ having hills in three colors, of which $k-1$ hills are of color $3$.
\end{enumerate}
\end{identity}
\begin{proof}  Dyck paths consider here have a simple recurrence, and it is easy to see that $g_3(n,k)$ satisfies this recurrence.
\end{proof}
\begin{remark} This array is also considered in~Koshy\cite[Chapter 14]{tk}.
\end{remark}

We derive one more relation between $g_2(n,k)$ and $g_3(n,k)$.
 \begin{proposition} We have
 \begin{equation}\label{cc2}g_2(n,k)=\sum_{i=0}^{k-1}{k\choose i}g_3(n-k,k-i).\end{equation}
\end{proposition}
\begin{proof}
Using~\cite[Proposition 2]{ja3}, we obtain
\[g_2(n,k)=\sum_{i_1+i_2+\cdots+i_k=n}C_{i_1-1}\cdot
C_{i_2-1}\cdots C_{i_k-1},\]
where the sum is taken over positive $i_t,(t=1,\ldots,k)$.
Replacing $i_t-1=j_t,(t=1,2,\ldots,k)$ we obtain
\[g_2(n,k)=\sum_{j_1+j_2+\cdots+j_k=n-k}C_{j_1}\cdot
C_{j_2}\cdots C_{j_k},\]
where the sum is taken over nonnegative $j_t$.
Note that in the case $k=n$ we have $g_2(n,n)=1.$ We consider the case $k<n$.
Assume that there are $i,(0\leq i\leq k-1)$ of $j_t$ which are equal $0$. Then
\[g_2(n,k)=\sum_{i=0}^{k-1}{k\choose i}\cdot\sum_{s_1+s_2+\cdots+s_{k-i}=n-k}
C_{s_1}\cdot C_{s_2}\cdots C_{s_{k-i}},\]
where $s_t>0,(t=1,2,\ldots,k-i)$ and $k-i\geq n-k$.
According Equation (\ref{sap}), we have
\[\sum_{s_1+s_2+\cdots+s_{k-i}=n-k}
C_{s_1}\cdot C_{s_2}\cdots C_{s_{k-i}}=g_3(n-k,k-i),\]
which proves the  statement.
\end{proof}
We next derive an explicit formula for $f_3(n)$.
\begin{proposition}The following formula holds:
\begin{equation*}f_3(n)={2n-1\choose n}.
\end{equation*}
\end{proposition}
\begin{proof} We have
\begin{align*}f_3(n)&=\frac{1}{n}\sum_{k=1}^nk{2n\choose n-k}=\frac{1}{n}\sum_{k=0}(n-k){2n\choose k}\\&=\sum_{k=0}^n{2n\choose k}-\sum_{k=1}^n\frac{k}{n}\cdot\frac{2n}{k}\cdot{2n-1\choose k-1}\\&=1+\sum_{k=1}^n\left[{2n-1\choose k}+{2n-1\choose k-1}\right]-2\sum_{k=1}^n{2n-1\choose k-1}\\&=\sum_{k=0}^n{2n-1\choose k}-\sum_{k=0}^{n-1}{2n-1\choose k}\\&={2n-1\choose n}.
\end{align*}
\end{proof}
\begin{remark} A combinatorial proof of this equation is given in~\cite[Proposition 3.1]{wsh}.

The preceding proof means that all results in our paper depend only on the fundamental properties of Fine and Catalan numbers.
\end{remark}

\section{Catalan array $g_4(n,k)$}
This case is considered in~\cite[Section 4]{ja3},
where the following results are obtained.
\begin{proposition}
\begin{enumerate}
\item
\begin{equation}\label{ehe}
g_4(n,k)=\frac{2^{n-k}}{n!}\sum_{i=1}^k(-1)^{k-i}{k\choose i}\cdot\prod_{j=0}^{n-1}(i+2j).
\end{equation}
\item
The value of $g_4(n,k)$ is the number of ternary  words of length $2n-1$,  having $k-1$ letters equal to $2$, and in all binary subwords the number of ones is greater by $1$ than the number of zeros.
 Also, each $2$ is both preceded and followed by a binary subword.
\item The value of $f_4(n)$  is the number of ternary words of length $2n-1$  in which $2$ is preceded and followed by a binary subword in which the number of ones is greater by $1$ than the number of zeros.
\end{enumerate}
\end{proposition}
As a consequence, we have the following Euler-type identities:
\begin{identity} The following sets have the same number of elements.
\begin{enumerate}
\item
The set of Dyck paths of semilength $n-1$ having hills in four  colors, of which $k-1$ in color $4$.
\item
The set of ternary  words of length $2n-1$,  having $k-1$ letters equal to $2$, and in all
binary subwords the number of ones is greater by $1$ than the number of zeros.
 Also, each $2$ is both preceded and followed by a binary subword.
\end{enumerate}
\end{identity}
\begin{identity}
\begin{enumerate}
\item
The set of Dyck paths of semilength $n-1$ having hills in four color.
\item
The set of ternary  words of length $2n-1$,  such that in all binary subwords the number of ones is greater by $1$ than the number of zeros, and  each $2$ is both preceded and followed by a binary subword.
\end{enumerate}
\end{identity}

\section{Some explicit formulas and identities}
From (\ref{CM}) and the fact that, for each integer $p$, we have
\begin{equation*}
L_n^p=\left(p^{i-j}{i-1\choose j-1}\right)_{n\times n},
\end{equation*}
 a mutually connection among different $g_m(n,k)$ is easy to obtain.

Up to now,  we have no an explicit formulas for $g_1(n,k)$.

In matrix form, we have $G_1(n)=G_3(n)L_n^{-2}.$ Hence, the following equation holds:
 \begin{equation}g_1(n,k)=\frac kn\cdot\sum_{i=k}^n(-2)^{i-k} {i\choose k}{2n\choose n-i}.\end{equation}
 Since $g_1(n,1)=\mathbb F_n$, we have the following identity for the Fine numbers:
\begin{identity}
\begin{equation*}\mathbb F_n=\frac 1n \cdot\sum_{i=1}^n(-2)^{i-1}\cdot i\cdot{2n\choose n-i}.
\end{equation*}
\end{identity}

We next prove the following binomial identity.
 \begin{identity}
 \begin{equation*}
{2n-2\choose n-1}=\sum_{i=1}^n(-1)^{i-1}i\cdot{2n\choose n-i}.
 \end{equation*}
 \end{identity}
 \begin{proof} We have
 \begin{equation*}f_1(n)=\sum_{k=1}^ng_1(n,k)=\frac 1n \sum_{k=1}^n\sum_{i=k}^nk(-2)^{i-k}{i\choose k}{2n\choose n-i}.\end{equation*}
 Changing the order of summation yields
 \begin{equation*} f_1(n)=\frac 1n\sum_{i=1}^n{2n\choose n-i}\cdot\sum_{k=1}^ik(-2)^{i-k}{i\choose k}.
 \end{equation*}
 \end{proof}
 Next, we have
 \begin{equation*}\sum_{k=1}^ik(-2)^{i-k}{i\choose k}=i\sum_{t=0}^{i-1}(-2)^t{i-1\choose t}=(-1)^{i-1}i.
 \end{equation*}

 Next, we write $g_2,g_3,g_4$ as alternating sums.

 Since $G_2(n)=G_3(n)\cdot L_n^{-1}$, we have
 \begin{equation}g_2(n,k)=\frac kn\cdot\sum_{i=k}^n(-1)^{i-k}{i\choose k}\cdot{2n\choose n-i}.\end{equation}
Comparing this equation and (\ref{g2}), we obtain the following identity:
\begin{identity} For $k>0$, we have
\begin{equation*}{2n-k-1\choose n}=\frac{n-k}{k}\sum_{i=k}^n(-1)^{i-k}{i\choose k}{2n\choose n-i}.\end{equation*}
\end{identity}

Also, since $g_2(n,1)=C_{n-1}$, we obtain the following identity for the Catalan numbers.
\begin{identity}
\begin{equation*}nC_{n-1}=\sum_{i=1}^n(-1)^{i-1}i\cdot{2n\choose n-i}.\end{equation*}
\end{identity}

We finish with two exotic identities. The first one consists of eigth items:
a sum, a product, two integers, a rising factorial, a falling factorial, and two binomial coefficients.
\begin{identity}
\begin{equation*}
k\cdot\prod_{i=1}^{n-k-1}(n+i)=(n-k-1)!\cdot\sum_{i=0}^{k-1}(k-i){k\choose i}{2n-2k\choose n+2k-i}.
\end{equation*}
\end{identity}
The identity is derived from (\ref{cc2}).

From $G_4(n)=G_3(n)L_n$, we obtain the identity consisting of ten items: an integer, two sums, a power of $-1$, a power of $2$, a falling  factorial, a rising factorial, and three binomial coefficients.
\begin{identity}
\begin{equation*}\sum_{i=k}^ni{i-1\choose k-1}{2n\choose n-i}=\frac{2^{n-k}}{(n-1)!}\sum_{i=1}^k(-1)^{k+i}{k\choose i} i(i+2)\cdots(i+2n-2).
\end{equation*}
\end{identity}


\begin{thebibliography}{9}
\bibitem{ja2} M. Janji\'c,  Some Formulas for Numbers
of Restricted Words, {\it  J. Integer Seq.} {\bf 20} (2017), Article 17.6.5
\bibitem{ja3} M. Janji\'c, Pascal Triangle and Restricted Words, arXiv:1705.02479.
\bibitem{tk} T. Koshy, {\it Catalan numbers with applications}, Oxford Univ. Press, (2009).
\bibitem{wsh} L. W. Shapiro,  A Catalan triangle, {\it Discrete Math.} 14 (1976) 83–90.
\bibitem{slo}  N. J. A. Sloane, The On-Line Encyclopedia of Integer Sequences, http://oeis.org.
\bibitem{rst} R. Stanley, {\it Catalan Numbers},  Cambridge University Press, (2015).
\end{thebibliography}
\end{document}